\RequirePackage{fix-cm}
\documentclass[smallextended]{svjour3}       
\smartqed  
\usepackage{graphicx}
\usepackage{color}
\begin{document}

\title{ A Parameterized Proximal Point Algorithm for Separable Convex Optimization
\thanks{The work was supported by the National Science Foundation of China under
grants 11671318 and 11571178, and the National Science Foundation of USA under grant 1522654.}
}

\author{Jianchao Bai \and Hongchao Zhang \and Jicheng Li
}


\institute{
    Jianchao Bai
    \at School of Mathematics and Statistics, Xi'an Jiaotong University, Xi'an 710049, P.R. China\\
    \email{{\tt bjc1987@163.com}}
\and
    Hongchao Zhang \at Department of Mathematics,
    Louisiana State University, Baton Rouge, LA 70803-4918, USA\\
    \email{{\tt hozhang@math.lsu.edu}}
\and
    Jicheng Li
    \at School of Mathematics and Statistics, Xi'an Jiaotong University, Xi'an 710049, P.R. China\\
    \email{{\tt jcli@mail.xjtu.edu.cn}}
}

\date{Received: date / Accepted: date}

\maketitle

\begin{abstract}
In this paper, we develop a Parameterized Proximal Point Algorithm (P-PPA) for solving a class of separable convex programming problems subject to linear and convex constraints. The proposed algorithm is provable to be globally convergent with a worst-case  $O(1/t)$ convergence rate, where $t$ denotes the iteration number. By properly choosing the algorithm parameters,  numerical experiments on solving a sparse optimization problem arising  from statistical learning show that our P-PPA could perform significantly better than other state-of-the-art methods, such as the Alternating Direction Method of Multipliers (ADMM) and the Relaxed Proximal Point Algorithm (R-PPA).

\keywords{Separable convex programming   \and  Proximal point algorithm \and Global convergence
         \and Statistical learning}
\subclass{  65C60\and 90C25 \and 90C33 }
\end{abstract}

\section{Introduction}
This article aims to study a novel parameterized proximal point method
for solving the following two-block separable convex optimization problem
\begin{equation} \label{Sec1-Prob}
\begin{array}{lll}
\min  &  f(x)+g(y)\\
\textrm{s.t. } &   Ax+By=c,\\
     &   x\in \mathcal{X}, y\in \mathcal{Y},
\end{array}
\end{equation}
where $f(x):\mathcal{R}^{m}\rightarrow\mathcal{R}$ and $\ g(y):\mathcal{R}^{n}\rightarrow\mathcal{R}$
are closed and proper convex functions, but not necessarily smooth; $A\in\mathcal{R}^{l\times m},
B\in\mathcal{R}^{l\times n}$ and $ c\in\mathcal{R}^{l}$ are given  matrices and vectors, respectively;
$\mathcal{X}\subset \mathcal{R}^{m}$ and $\mathcal{Y}\subset \mathcal{R}^{n} $ are closed convex sets.
Throughout the paper, we   assume
the solution set of the problem (\ref{Sec1-Prob}) is nonempty and
the matrices $A$ and $B$ have full column rank.

It is well-known in the literature that proximal point methods are a class of
benchmark methods for solving the problem (\ref{Sec1-Prob}). The Proximal Point
Algorithm (PPA) was originally proposed for solving monotone operator inclusion problems
\cite{Moreau1965,Martinet1970} and then became popularized to   convex programmings
by Rockafellar \cite{Rockafellar1976} and  Eckstein\cite{Eckstein1993}.
As demonstrated
in \cite{Rockafellar1976}, the augmented Lagrangian method \cite{Powell1969} for solving the  problem  (\ref{Sec1-Prob})
is actually an application of PPA to its dual problem. And the recently very popular Alternating Direction Method of Multipliers (ADMM)
can be also regarded as another special variant of PPA to the dual problem \cite{EcksteinBertsekas1992}. Due to its simplicity for implementation, efficiency and strong theoretical
background, the PPA has attracted extensive researches in recent years for solving structured convex optimization problems,
especially by the optimizers from the areas involving lots of structured data,
such as compressed sensing, image processing and machine learning, etc.

There is very rich literature on PPA.  Combettes and  Pennanen \cite{CombettesPennanen2004} showed some conditions  for the viability and  weak convergence of an inexact Relaxed PPA (R-PPA) for finding a common zero of countably many cohypomonotone operators in  Hilbert space. Later, based on the closed-form expressions for the proximity operators \cite{CombettesPesquet2007}, Combettes et al.\cite{CombettesPustelnik2011}  still derived expressions of
new proximity operators in product spaces and  presented an extension of  PPA for solving the  multicomponent signal/image processing
problems. Recently, PPA was extensively studied for  a class of multi-criteria optimization problems with the difference of convex objective functions, whose efficiency was demonstrated by testing a multi-period portfolio minimization problem, see \cite{JiGoh2016} for more details. More recently, by using proximal regularization techniques and  partially parallel splitting schemes, Wang et al.\cite{WangHe2017}  developed a proximal partially parallel splitting method for a multi-objective convex minimization problem. For an extensive review on PPA, one may refer \cite{GuHeYuan2014,HeYuanZhang2013} and the references therein.
In what follows, we simply mention a few works closely related to the development of our proposed method.
He et al.\cite{HeYuanZhang2013}  investigated a customized application of the classical PPA for the convex
programming with linear constraints, where some image processing problems were
tested to show the efficiency of their method. Cai et al.\cite{CaiGuHeYuan2013} also proposed a R-PPA for solving  (\ref{Sec1-Prob}) and analyzed its global convergence with a worst-case linear convergence
rate.
Based on the results of \cite{CaiGuHeYuan2013,GuHeYuan2014,HeYuanZhang2013}, Ma and Ni \cite{MaNi2016}
recently revisited the application of PPA for solving the basis pursuit and  matrix completion problem.
Our new proposed Parameterized PPA (\textbf{P-PPA}) can be actually regarded as more general extensions of the
algorithms developed in \cite{HeYuanZhang2013} and \cite{MaNi2016}  which does not make use of
the separable structure of the objective function in (\ref{Sec1-Prob}).

 Major contributions of this paper are summarized in the following. Firstly, the proximal matrix  in our  proposed
P-PPA is more general and flexible than those in the previous
work \cite{HeYuanZhang2013,MaNi2016}, due to more induced parameters to take consideration of the problem structure  instead of a unique objective function.
Secondly, by properly choosing the algorithm parameters, the new P-PPA could significantly outperform some
state-of-the-art methods, such as ADMM \cite{BoydChu2010} and R-PPA \cite{GuHeYuan2014},
for solving the separable convex optimization, especially when the problem size is large and high accurate solutions are required.

The remaining parts  are organized as follows. In Section 2,  we characterize the solution
of the problem (\ref{Sec1-Prob}) as the solution of proper variational inequalities and review the unified framework of PPA.
In Section 3, we derive the new P-PPA and discuss its global convergence and worst-case convergence rate in an ergodic sense.
At the end of Section 3, we still present a P-PPA with a relaxation step. Some preliminary numerical experiments are
performed in Section 4 for comparing our proposed methods with two benchmark methods.
We finally conclude the paper in Section 5.

\section{Preliminaries}
In this section,  we first introduce some necessary notations used throughout the paper. Then, we characterize the solution of the problem (\ref{Sec1-Prob}) by  the aid of an equivalent variational inequality.
Similar approaches have been widely used in the literature,  e.g. \cite{HeYuanZhang2013,HeMaYuan2016}.

For the sake of convenience,
let $\mathcal{R}$,  $\mathcal{R}^n$ and
 $\mathcal{R}^{n\times m}$   denote the set of real numbers, the set of $n$ dimensional real column vectors
and the set of $n\times m$ real matrices, respectively.
For any $x, y \in \mathcal{R}^n$, $\langle x, y\rangle=x^Ty$ denotes the standard inner product in $\mathcal{R}^n$
and  $\|x\|= \sqrt{\langle x, x\rangle}$ is the Euclidean norm. Given any symmetric positive definite matrix
$G\in \mathcal{R}^{n\times n}$,  the weighted norm $\|x\|_G = \sqrt{\langle x, G x\rangle}$.  In addition,
we use   $\textbf{I}$ and $\textbf{0}$
to  stand for the identity matrix and the zero vector with proper dimension, respectively.

The following basic lemma given in \cite{HeMaYuan2016} will be used as a tool for analyzing the primal-dual solution pair of the problem (\ref{Sec1-Prob}).
\begin{lemma}\label{opt-1}
Let $\varphi:\mathcal{R}^m\rightarrow \mathcal{R}$ and
$\psi:\mathcal{R}^m\rightarrow \mathcal{R}$ be two convex functions defined on a closed convex
set $\Omega\subset \mathcal{R}^m$ and $\psi$ is differentiable. Suppose the solution set
$\Omega^* = \arg\min\limits_{x\in\Omega}\{\varphi(x)+\psi(x) \}$ is nonempty. Then we have
\[
x^* \in \Omega^* \ \mbox{ if and only if } \ x^*\in \Omega, \  \varphi(x)-\varphi(x^*)
+\left\langle x-x^*, \nabla \psi(x^*)\right\rangle\geq 0, \forall x\in\Omega.
\]
\end{lemma}

Now, given any $ \tau\in \mathcal{R}\backslash \{0\}$, that is $\tau\neq 0$, a Lagrangian function
of problem (\ref{Sec1-Prob}) can be written as
\begin{equation} \label{Sec2-Lagra}
L(x,y,\lambda)= f(x)+g(y)-\left\langle\lambda, \tau(Ax+By-c)\right\rangle,
\end{equation}
where $\lambda\in \mathcal{R}^l$ is the Lagrange multiplier.
Then, for any primal-dual solution pair $(x^*,y^*,\lambda^*)$ of  (\ref{Sec1-Prob}),
we have
\[
L(x^*,y^*,\lambda) \leq L(x^*,y^*,\lambda^*)\leq L(x,y,\lambda^*),
\]
which is equivalent to
\[
 \left \{\begin{array}{l}
x^*=\arg\min\{f(x)-\langle\lambda, \tau Ax\rangle| x\in\mathcal{X}\},\\
y^*=\arg\min\{g(y)-\langle\lambda, \tau By\rangle|\ y\in\mathcal{Y}\},\\
\lambda^*=\arg\max\{-\left\langle\lambda, \tau (Ax+By-c)\right\rangle|\ \lambda\in\mathcal{R}^l\}.
\end{array}\right.
\]
By applying  Lemma \ref{opt-1}, the optimality conditions of the above equations are
\begin{equation} \label{Sec2-Lagra-1}
 \left \{\begin{array}{lll}
x^*\in\mathcal{X}, & f(x)-f(x^*)+ \left\langle x-x^*, -\tau A^T\lambda^*\right\rangle\geq 0, &x\in\mathcal{X},\\
y^*\in\mathcal{Y}, &  g(y)-g(y^*)+ \left\langle y-y^*, -\tau B^T\lambda^*\right\rangle\geq 0,&y\in\mathcal{Y},\\
\lambda^*\in\mathcal{R}^l, &  \left\langle \lambda-\lambda^*, \tau (Ax^*+By^*-c)\right\rangle\geq 0,&\lambda\in\mathcal{R}^l,
\end{array}\right.
\end{equation}
which can be rewritten as a variational inequality (\textbf{VI})
\begin{equation} \label{opt-vi-1}
\textrm{VI}(\phi, \mathcal{J}, \mathcal{M}):\quad \phi(u)- \phi(u^*) + \left\langle w-w^*,
\mathcal{J}(w^*)\right\rangle \geq 0,\quad \forall w\in \mathcal{M},
\end{equation}
where
\[
\phi(u)=f(x) +g(y),\quad \mathcal{M}=\mathcal{X}\times \mathcal{Y}\times \mathcal{R}^l,
\]
\[
u=\left(\begin{array}{c}
x\\  y\\
\end{array}\right),\
w=\left(\begin{array}{c}
x\\ y\\  \lambda
\end{array}\right) \ \mbox{and} \
\mathcal{J}(w)=\tau\left(\begin{array}{c}
-A^T\lambda\\
-B^T\lambda\\
 A x + B y - c
\end{array}\right).
\]
Clearly, the solution set of $\textrm{VI}(\phi, \mathcal{J}, \mathcal{M})$, denoted by $\mathcal{M}^*$,
is nonempty by the assumption of nonempty solution set of the problem (\ref{Sec1-Prob}).
Since the affine mapping $\mathcal{J}$ is skew-symmetric, we can obtain
\begin{equation} \label{Sec2-Jw}
\left\langle w-\widehat{w}, \mathcal{J}(w)\right\rangle= \left\langle w-\widehat{w}, \mathcal{J}(\widehat{w})\right\rangle,
\quad \forall\ w, \widehat{w}\in \mathcal{M}.
\end{equation}
Hence, the variational inequality (\ref{opt-vi-1}) is also rewritten as
\begin{equation} \label{opt-vi-2}
\phi(u)- \phi(u^*) + \left\langle w-w^*,
\mathcal{J}(w)\right\rangle \geq 0,\quad \forall w\in \mathcal{M}.
\end{equation}

When the proximal point algorithms are applied to solve the variational inequality
 (\ref{opt-vi-2}) or equivalently (\ref{opt-vi-1}), they often take the following unified approach: at the $k$-th iteration,
find iterate $w^{k+1}$ satisfying
\begin{equation} \label{Sec2-PPA}
\phi(u)- \phi(u^{k+1}) + \left\langle w-w^{k+1}, \mathcal{J}(w^{k+1})+G\left(w^{k+1}-w^k\right)\right\rangle \geq 0,\ \
\forall w\in \mathcal{M},
\end{equation}
 where the above matrix $G\in \mathcal{R}^{(m+n+l)\times (m+n+l)}$ called proximal matrix is a positive definite, and
\[
u^{k+1}=\left(\begin{array}{c}
x^{k+1}\\  y^{k+1}\\
\end{array}\right),\
w^{k+1}=\left(\begin{array}{c}
x^{k+1}\\ y^{k+1}\\  \lambda^{k+1}
\end{array}\right), \
\mathcal{J}(w^{k+1})=\tau\left(\begin{array}{c}
-A^T\lambda^{k+1}\\
-B^T\lambda^{k+1}\\
 A x^{k+1} + B y^{k+1} - c
\end{array}\right).
\]
Obviously, different choices of $G$ would result in different
proximal point algorithms.
We would provide a new choice of $G$ for our proposed P-PPA.

\section{Main results}
In this section, we first develop the P-PPA for solving (\ref{Sec1-Prob}) in detail and discuss its convergence properties.
Then, it is  straightforward to extend the method to the case with a relaxation step.

\subsection{Development of P-PPA with convergence}
Mainly motivated by the  proximal matrix  of PPA  in Eq.(2.5) of \cite{HeYuanZhang2013} and Eq.(3.1) of \cite{MaNi2016}, we would design the matrix $G$ in (\ref{Sec2-PPA}) having the following structure:
\begin{equation} \label{Sec3-G}
G=\left[\begin{array}{cc|c}
\left(\sigma+\frac{\varepsilon^2-1}{s}\right)A^TA&  & -\varepsilon A^T\\
  &  \left(\rho+\frac{\tau^2-1}{s}\right)B^TB &-\tau B^T\\ \hline
-\varepsilon A & -\tau B& s\textbf{I}
\end{array}\right],
\end{equation}
where  $(\sigma, \rho, s,\tau,\varepsilon)$ are parameters satisfying
\begin{equation} \label{Sec3-para-region}
  s>0,\ \sigma >\frac{1}{s},\ (\sigma s-1)(\rho s-1) -\tau^2\varepsilon^2>0, \varepsilon\in
 \mathcal{R} \mbox{ and } \tau\in \mathcal{R}\backslash \{0\}.
\end{equation}
For convenience, let us define
\begin{equation}\label{sigma-rho-bar}
\bar{\sigma} = \sigma+\frac{\tau^2-1}{s} \quad \mbox{ and } \quad
\bar{\rho}=\rho+\frac{\tau^2-1}{s}.
\end{equation}
Then, from later analysis we can see that
 ${1}/{s}$ would play a role of penalty parameter for the equality constraint
 of (\ref{Sec1-Prob}), while
 $\bar{\sigma}$ and $ \bar{\rho}$ can be regarded as the proximal parameters
as those used in the customized PPA \cite{HeYuanZhang2013}.

The following lemma ensures that under proper conditions of the parameters, $G$ is a positive  definite matrix.
\begin{lemma}\label{G-psd}
Suppose that the matrices $A$ and $B$ have full column rank.
For any  $(\sigma, \rho, s,\tau,\varepsilon)$ satisfying (\ref{Sec3-para-region}), the matrix $G$ defined in (\ref{Sec3-G}) is
positive definite.
\end{lemma}
\begin{proof}
Clearly, the matrix $G$ is symmetric and can be decomposed into
\[
G=D^TG_0D,
\]
where
\begin{equation} \label{Sec3-DG0}
D=\left[\begin{array}{ccc}
A&  &  \\
  &  B & \\
  &  & \textbf{I}
\end{array}\right] \ \textrm{and}
\quad G_0=\left[\begin{array}{cc|c}
\left(\sigma+\frac{\varepsilon^2-1}{s}\right)\textbf{I}&  & -\varepsilon \textbf{I}\\
  &  \left(\rho+\frac{\tau^2-1}{s}\right)\textbf{I} &-\tau \textbf{I}\\ \hline
-\varepsilon \textbf{I} & -\tau \textbf{I}& s\textbf{I}
\end{array}\right].
\end{equation}
By the full column rank assumption of $A$ and $B$, the matrix $G$ is positive definite
if and only if $G_0$ is positive definite.
Noting that
\[
\left[\begin{array}{ccc}
\textbf{I} & & \frac{\varepsilon}{s}\textbf{I}\\
& \textbf{I}&\frac{\tau}{s}\textbf{I}\\
& & \textbf{I}
\end{array}\right] \ G_0 \ \left[\begin{array}{ccc}
\textbf{I} & & \frac{\varepsilon}{s}\textbf{I}\\
& \textbf{I}&\frac{\tau}{s}\textbf{I}\\
& & \textbf{I}
\end{array}\right]^T\\
 = \left[\begin{array}{cc|c}
\left(\sigma-\frac{1}{s}\right)\textbf{I} & -\frac{\tau\epsilon}{s}\textbf{I}&  \\
-\frac{\tau\epsilon}{s}\textbf{I}& \left(\rho-\frac{1}{s}\right)\textbf{I}& \\ \hline
& & s\textbf{I}
\end{array}\right] = : \widetilde{G}_0.
\]
Hence, $G_0$ is positive definite if and only if $\widetilde{G}_0$ is positive
definite, which is guaranteed if condition (\ref{Sec3-para-region}) holds.
Therefore, the proof is completed.
  $\  \ \ \diamondsuit$
\end{proof}

In what follows, we develop our  P-PPA in detail. Substituting the matrix $G$ into  (\ref{Sec2-PPA}),
we have  $\lambda^{k+1}\in\mathcal{R}^l$ and
\[
 \langle \lambda-\lambda^{k+1}, R_{\lambda} \rangle \geq 0,\ \forall\lambda\in\mathcal{R}^l ,
\]
where
\begin{eqnarray*}
R_{\lambda} & = & \tau\left( A x^{k+1} + B y^{k+1} - c\right)-\varepsilon A\left(x^{k+1}-x^k\right)
-\tau B\left(y^{k+1}-y^k\right) \\
&& +s\left(\lambda^{k+1}-\lambda^{k}\right).
\end{eqnarray*}
Hence, we have $ R_{\lambda} = 0$ which leads to
\begin{equation} \label{Sec3-labdaupdate}
 \lambda^{k+1}=\lambda^{k}-\frac{1}{s}\left[(\tau-\varepsilon)Ax^{k+1}
+\varepsilon Ax^k+\tau By^k-\tau c\right].
\end{equation}
Meanwhile, by (\ref{Sec2-PPA}) and (\ref{Sec3-labdaupdate}), we also have
\begin{equation} \label{Sec3-x-subp}
x^{k+1}\in\mathcal{X},\quad f(x)-f(x^{k+1})+\left \langle
x-x^{k+1},
R_x\right\rangle\geq 0, \quad \forall x\in\mathcal{X},
\end{equation}
where
\begin{eqnarray*}
R_x&=&  -\tau A^T\lambda^{k+1}-\varepsilon A^T\left(\lambda^{k+1}-\lambda^{k}\right)
+\left(\sigma+\frac{\varepsilon^2-1}{s}\right)A^TA\left(x^{k+1}-x^k\right)  \\
&=& -(\tau+\varepsilon)A^T\left\{\lambda^{k}-\frac{1}{s}\left[(\tau-\varepsilon)Ax^{k+1}
+\varepsilon Ax^k+\tau By^k-\tau c\right]\right\}\\
& & +\varepsilon A^T\lambda^{k}+\left(\sigma+\frac{\varepsilon^2-1}{s}\right)A^TA\left(x^{k+1}-x^k\right)\\
&=&-\tau A^T\bar{\lambda}^{k}+\bar{\sigma}A^TA\left(x^{k+1}-x^k\right)
\end{eqnarray*}
with $\bar{\sigma}$ being defined in (\ref{sigma-rho-bar}) and
\begin{equation} \label{Sec3-x-labmad}
\bar{\lambda}^k=\lambda^{k}-\frac{\tau+\varepsilon}{s}\left(Ax^k+By^k-c\right).
\end{equation}
Notice that by (\ref{Sec3-para-region}), we get $\bar{\sigma} = (\sigma s + \tau^2 -1)/s >0$.
Hence, by Lemma \ref{opt-1}, $x^{k+1}$ is the solution of the following optimization problem
\begin{eqnarray} \label{Sec3-xsubk1}
x^{k+1} &=&\arg\min\limits_{x\in \mathcal{X}}
\left\{
f(x)+ \frac{\bar{\sigma}}{2}\left\|A(x-x^k)\right\|^2
 -\left\langle Ax,\tau \bar{\lambda}^k\right\rangle
\right\} \nonumber \\
&=& \arg\min\limits_{x\in \mathcal{X}}\left\{f(x)+ \frac{\bar{\sigma}}{2}
\left\|A(x-x^k)-\frac{\tau}{\bar{\sigma}}\bar{\lambda}^k \right\|^2\right\}.
\end{eqnarray}
Similarly, we can also derive from (\ref{Sec2-PPA}) and (\ref{Sec3-labdaupdate}) that
\begin{equation} \label{Sec3-y-subp}
y^{k+1}\in \mathcal{Y},\quad g(y)-g(y^{k+1})+\left \langle
y-y^{k+1},
R_y\right\rangle\geq 0, \quad \forall y\in\mathcal{Y},
\end{equation}
where
\begin{eqnarray*}
R_y&=&  -\tau B^T\lambda^{k+1}-\tau B^T\left(\lambda^{k+1}-\lambda^{k}\right)
+\left(\rho+\frac{\tau^2-1}{s}\right)B^TB\left(y^{k+1}-y^k\right)  \\
&=& -2\tau B^T\left\{\lambda^{k}-\frac{1}{s}\left[(\tau-\varepsilon)Ax^{k+1}+\varepsilon Ax^k+\tau By^k-\tau c\right]\right\}\\
& & +\tau B^T\lambda^{k}+\left(\rho+\frac{\tau^2-1}{s}\right)B^TB\left(y^{k+1}-y^k\right) \\
&=&-\tau B^T\bar{\lambda}^{k+\frac{1}{2}}+ \bar{\rho} B^TB\left(y^{k+1}-y^k\right),
\end{eqnarray*}
and
\begin{equation} \label{lambda-half}
\bar{\lambda}^{k+\frac{1}{2}} = \lambda^k-\frac{2}{s}\left[(\tau-\varepsilon)Ax^{k+1}+\varepsilon Ax^k+\tau By^k-\tau c\right].
\end{equation}
Furthermore, it follows from (\ref{Sec3-x-labmad}) and (\ref{lambda-half}) that
\begin{eqnarray} \label{Sec3-y-labmad}
\bar{\lambda}^{k+\frac{1}{2}} & = & \bar{\lambda}^k + \frac{\tau+\varepsilon}{s}\left(Ax^k+By^k-c\right) \nonumber \\
& & -\frac{2}{s}\left[(\tau-\varepsilon)Ax^{k+1}+\varepsilon Ax^k+\tau By^k-\tau c\right] \nonumber \\
& =& \bar{\lambda}^k-\frac{\tau - \varepsilon}{s}\left[A(2x^{k+1}-x^k)+By^k- c\right].
\end{eqnarray}
Hence, by Lemma \ref{opt-1}, $y^{k+1}$ is the solution of the following optimization problem
\begin{eqnarray}\label{Sec3-ysubk1}
y^{k+1} &= &\arg\min\limits_{y\in \mathcal{Y}}
\left\{ g(y)+ \frac{\bar{\rho}}{2} \left\|B(y-y^k)\right\|^2
 -\left\langle By,\tau\bar{\lambda}^{k+\frac{1}{2}}\right\rangle \right\} \nonumber \\
& = & \arg\min\limits_{y\in \mathcal{Y}}\left\{g(y)+ \frac{\bar{\rho}}{2}
\left\|B(y-y^k)-\frac{\tau}{\bar{\rho}}\bar{\lambda}^{k+\frac{1}{2}} \right\|^2\right\}.
\end{eqnarray}
In addition, it follows from (\ref{Sec3-labdaupdate})  and (\ref{Sec3-x-labmad}) that
\begin{eqnarray}\label{Sec3-x-labmad-1}
\bar{\lambda}^{k+1}
&= &\lambda^{k+1}- \frac{\tau+\varepsilon}{s}\left(Ax^{k+1}+By^{k+1}-c\right) \nonumber \\
& =& \lambda^k -\frac{1}{s}\left[(\tau-\varepsilon)Ax^{k+1} +\varepsilon Ax^k+\tau By^k-\tau c\right] \nonumber \\
& & - \frac{\tau+\varepsilon}{s}\left(Ax^{k+1}+By^{k+1}-c\right) \nonumber \\
& =&  \bar{\lambda}^k + \frac{\tau+\varepsilon}{s}\left(Ax^k+By^k-c\right) - \frac{\tau+\varepsilon}{s}\left(Ax^{k+1}+By^{k+1}-c\right) \nonumber \\
&  &  -\frac{1}{s}\left[(\tau-\varepsilon)Ax^{k+1} +\varepsilon Ax^k+\tau By^k-\tau c\right] \nonumber \\
& = & \bar{\lambda}^{k}-\frac{\tau}{s} ( A x^{k+1} + B y^{k+1} -c ) \nonumber \\
&& -\frac{1}{s}\left[ \tau A (x^{k+1} - x^k) + \varepsilon B(y^{k+1} - y^k) \right].
\end{eqnarray}

Summarizing all the above discussions, we propose P-PPA as the following algorithm:

\vskip4mm
\hrule\vskip2mm
\noindent {\bf Algorithm 1} (P-PPA for solving Problem (\ref{Sec1-Prob}))
\vskip1.5mm\hrule\vskip2mm
\noindent 1\ \ Choose parameters $(\sigma, \rho, s,\tau, \varepsilon)$ satisfying (\ref{Sec3-para-region});\\
2\ \ Initialize $(x^0, y^0, \lambda^0)\in \mathcal{R}^m\times \mathcal{R}^n\times\mathcal{R}^l$ and $r^0 = Ax^0+By^0-c$;\\
3\ \ $\bar{\lambda}^0 = \lambda^{0}-\frac{\tau+\varepsilon}{s} r^0$; \\
4\ \ \textbf{For} $k=0,1,\cdots,$ \textbf{do}\\
5\ \ \quad $x^{k+1}=\arg\min\limits_{x\in \mathcal{X}}\left\{f(x)+ \frac{\bar{\sigma}}{2}
\left\|A(x-x^k)-\frac{\tau}{\bar{\sigma}} \bar{\lambda}^{k} \right\|^2\right\}$;\\
6\ \ \quad  $\bar{\lambda}^{k+\frac{1}{2}}=\bar{\lambda}^k-\frac{\tau - \varepsilon}{s}\left[A(2x^{k+1}-x^k)+By^k- c\right]$;\\
7\ \ \quad  $y^{k+1}=\arg\min\limits_{y\in \mathcal{Y}}\left\{g(y)+ \frac{\bar{\rho}}{2}
\left\|B(y-y^k)-\frac{\tau}{\bar{\rho}} \bar{\lambda}^{k+\frac{1}{2}} \right\|^2\right\}$;\\
8\ \ \quad  $r^{k+1} = A x^{k+1} + B y^{k+1} -c$; \\
9\ \ \quad  $\bar{\lambda}^{k+1}=\bar{\lambda}^{k}-\frac{\tau}{s} r^{k+1}
-\frac{1}{s}\left[ \tau A (x^{k+1} - x^k) + \varepsilon B(y^{k+1} - y^k) \right]$.\\
\hrule\vskip4mm

For the above P-PPA, we have the following remarks.
\begin{remark} By taking $\tau=\varepsilon=1,$ the matrix $G$ in (\ref{Sec3-G}) would become
\[
G=\left[\begin{array}{ccc}
\sigma A^TA& & -A^T\\
 &  \rho B^TB &- B^T\\
- A & - B& s\textbf{I}
\end{array}\right],
\]
where the parameters $(\sigma,\rho,s)$ satisfy
\[
s>0,\quad \sigma >\frac{1}{s},\quad (\sigma s-1)(\rho s-1)>1.
\]
In such case,  Algorithm 1 would be reduced to
\[
 \left \{
 \begin{array}{lll}
  \textrm{Let} \ \bar{\lambda}^{0} = \lambda^{0}-\frac{2}{s}\left(Ax^0+By^0-c\right);\\
\textrm{For}\ k=0,1,\cdots, \textrm{do}\\
\quad x^{k+1}=\arg\min\limits_{x\in \mathcal{X}}\left\{f(x)+ \frac{\bar{\sigma}}{2}
\left\|A(x-x^k)-\frac{\tau}{\bar{\sigma}} \bar{\lambda}^{k} \right\|^2\right\};\\
\quad y^{k+1}=\arg\min\limits_{y\in \mathcal{Y}}\left\{g(y)+ \frac{\bar{\rho}}{2}
\left\|B(y-y^k)-\frac{\tau}{\bar{\rho}}\bar{\lambda}^{k} \right\|^2\right\};\\
\quad \bar{\lambda}^{k+1}= \bar{\lambda}^{k} - \frac{1}{s}\left[ A (2 x^{k+1} - x^k) + B (2 y^{k+1} - y^k) -c \right].
\end{array}\right.
\]
This algorithm can be considered as a direct extension of the customized PPA  \cite{HeYuanZhang2013},
which only considers problem with one block structure, to solve the
two-block structured problem (\ref{Sec1-Prob}).
\end{remark}

\begin{remark}  \label{agaga}
The freedom of choosing the parameters $(\sigma,\rho,s,\tau, \varepsilon)$ would allow P-PPA to have more flexibility to select
the proximal parameters $\bar{\sigma}$ and $\bar{\rho}$.
However, note that   they are  proportional to the parameters $(\sigma, \tau,1/s)$ and $(\rho, \tau,1/s)$, respectively.
As commented in the final part of \cite{HeXuYuan2016}, if they are too large, then slow convergence will occur in terms of  solving the  subproblems, which  can significantly affect the
 overall efficiency of the algorithm.
\end{remark}

Next, we discuss the convergence properties of P-PPA in a more general proximal point setting given in (\ref{Sec2-PPA}).
\begin{lemma}\label{Contract}
The sequence $\{w^{k+1}\}$ generated by Algorithm 1 satisfies
\begin{equation} \label{Sec3-IEq}
\|w^{k+1}-w^*\|_G^2\leq \|w^{k}-w^*\|_G^2- \|w^{k+1}-w^k\|_G^2,\quad  \forall w^*\in \mathcal{M}^*.
\end{equation}
\end{lemma}
\begin{proof} Setting $w=w^*$ in (\ref{Sec2-PPA}), we have
\[
\phi(u^*)- \phi(u^{k+1}) + \left\langle w^*-w^{k+1}, \mathcal{J}(w^{k+1})+G\left(w^{k+1}-w^k\right)\right\rangle \geq 0,
\]
which, by (\ref{opt-vi-1})-(\ref{Sec2-Jw}), implies
\[
\left\langle w^*-w^{k+1},G\left(w^{k+1}-w^k\right)\right\rangle \geq 0.
\]
By the above inequality, we obtain
\[
\begin{array}{lll}
\|w^k-w^*\|_G^2&=&\|w^k-w^{k+1}+ w^{k+1}-w^*\|_G^2\\
&\geq &\|w^k-w^{k+1}\|_G^2+\|w^{k+1}-w^*\|_G^2,
\end{array}
\]
which immediately gives the inequality (\ref{Sec3-IEq}).  $\  \ \ \diamondsuit$
\end{proof}

Based on the above lemma, we have the following global convergence theorem.
\begin{theorem}\label{global-converge}
Suppose that the condition (\ref{Sec3-para-region}) holds and the sequence $\{w^{k+1}\}$ is generated by  Algorithm 1.
Then, there exists a $w^\infty \in \mathcal{M}^*$ such that
\begin{equation} \label{point-converge}
\lim_{k \to \infty} w^k = w^\infty.
\end{equation}
\end{theorem}
\begin{proof}
Since condition (\ref{Sec3-para-region}) holds, we have by Lemma \ref{G-psd} that $G$ is positive definite.
Then, it follows from (\ref{Sec3-IEq}) that $\{w^{k}\}$ is bounded and
\begin{equation} \label{diff-zero}
\lim_{k\to \infty} \| w^k-w^{k+1}\|=0.
\end{equation}
Let $w^\infty$ be any accumulation point of $\{w^{k}\}$. By taking a subsequence of $w^k$ in (\ref{Sec2-PPA}) if necessary,
it follows from (\ref{diff-zero}) that
\[
\phi(u)-\phi(u^\infty) + \left\langle w-w^\infty, \mathcal{J}(w^\infty)\right\rangle
\geq 0,\ \forall w\in \mathcal{M}.
\]
Hence, $w^\infty \in \mathcal{M}^*$. So, by (\ref{Sec3-IEq}) again, we have
\[
\|w^k-w^\infty\|_G \leq \|w^j-w^\infty\|_G \quad \mbox{ for all } k \ge j.
\]
Then, it follows from  $w^\infty$ being an accumulation point that (\ref{point-converge}) holds.
 $\ \ \ \diamondsuit$
\end{proof}

Now, we establish the  worst-case $\mathcal{O}(1/t)$ ergodic convergence rate
for solving the variational inequality (\ref{opt-vi-2}). Let
\begin{equation}\label{erg-iterate}
\textbf{w}_t :=\frac{1}{t+1}\sum_{k=0}^{t}w^{k+1}\quad \mbox{and} \quad \textbf{u}_t :=\frac{1}{t+1}\sum_{k=0}^{t}u^{k+1}.
\end{equation}

\begin{theorem}\label{conv-rate}
Suppose that the condition (\ref{Sec3-para-region}) holds and the sequence $\{w^{k+1}\}$ is generated by  Algorithm 1.
Then, we have
\begin{equation} \label{erg-rate}
\phi(\textbf{u}_t)-\phi(u)+ \left\langle \textbf{w}_t-w, \mathcal{J}(w)\right\rangle \leq \frac{1}{2(t+1)}\left\|w^0-w\right\|_G^2,
\ \forall\ w\in \mathcal{M}.
\end{equation}
\end{theorem}
\begin{proof}
By (\ref{Sec2-PPA}) and the property (\ref{Sec2-Jw}), it holds that
\begin{equation}\label{Sec2-FianlVI}
\phi(u)-\phi(u^{k+1})+ \left\langle w-w^{k+1}, \mathcal{J}(w)\right\rangle \geq \left\langle w^{k+1}-w, G(w^{k+1}-w^k)\right\rangle,
\forall\ w\in \mathcal{M}.
\end{equation}
Then, applying the identity
\[
2\langle a-b, G(a-c)\rangle=\|a-c\|_G^2+\|a-b\|_G^2- \|c-b\|_G^2
\]
with
\[
a=w^{k+1},\quad b=w,\quad c=w^k,
\]
we can obtain
\[
 \begin{array}{lll}
\left\langle w^{k+1}-w, G(w^{k+1}-w^k)\right\rangle&=&\frac{1}{2}\left( \left\|w^{k+1}-w^k\right\|_G^2+ \left\|w^{k+1}-w\right\|_G^2
- \left\|w^{k}-w\right\|_G^2\right)\\
&\geq&\frac{1}{2}\left( \left\|w^{k+1}-w\right\|_G^2- \left\|w^{k}-w\right\|_G^2\right),
\end{array}
\]
which together with (\ref{Sec2-FianlVI}) imply
\[
\phi(u)-\phi(u^{k+1})+ \left\langle w-w^{k+1}, \mathcal{J}(w)\right\rangle +\frac{1}{2}\left\|w^{k}-w\right\|_G^2
\geq \frac{1}{2}\left\|w^{k+1}-w\right\|_G^2.
\]
Summing the above inequality over $k=0,1,\cdots,t$, we get
\[
(t+1)\phi(u)-\sum\limits_{k=0}^{t}\phi(u^{k+1})+ \left\langle (t+1)w-\sum\limits_{k=0}^{t}w^{k+1}, \mathcal{J}(w)\right\rangle
+\frac{1}{2}\left\|w^{0}-w\right\|_G^2 \geq 0,
\]
which by the definition of $\textbf{w}_t$ in (\ref{erg-iterate}) gives
\begin{equation}\label{Sec2-FianlEq}
\frac{1}{t+1}\sum\limits_{k=0}^{t}\phi(u^{k+1})-\phi(u)+ \left\langle \textbf{w}_t-w, \mathcal{J}(w)\right\rangle
\leq\frac{1}{2(t+1)}\left\|w^{0}-w\right\|_G^2.
\end{equation}
By the convexity of the function $\phi(u)$ and the definition of $\textbf{u}_t$ in (\ref{erg-iterate}), we have
\[
\phi(\textbf{u}_t)\leq \frac{1}{t+1}\sum_{k=0}^{t}\phi(u^{k+1}).
\]
Then, (\ref{erg-rate}) is true by
substituting the above inequality into (\ref{Sec2-FianlEq}).  $\ \ \ \diamondsuit$
\end{proof}

\subsection{P-PPA with a relaxation step}

Applying a relaxation step for PPA is a standard technique to accelerate its convergence
\cite{GuHeYuan2014,Golsshtein1979,HeYuanZhang2013}.
Combining the PPA approach  (\ref{Sec2-PPA}) together with a relaxation step would give the following procedure:
 at the $k$-th iteration, find $\widetilde{w}^{k+1}$ satisfying
\begin{equation} \label{relax-PPA}
\phi(u)- \phi(\widetilde{u}^{k+1}) + \left\langle w- \widetilde{w}^{k+1}, \mathcal{J}(\widetilde{w}^{k+1})+G\left(\widetilde{w}^{k+1}-w^k\right)\right\rangle \geq 0,\ \
\forall w\in \mathcal{M},
\end{equation}
and then let
\[
 w^{k+1} = w^k + \gamma (\widetilde{w}^{k+1} - w^k),
\]
where $\gamma \in (0,2)$ is the relaxation parameter. When $\gamma =1$, the above relaxed PPA
will be reduced to the standard PPA. For the relaxed PPA, analogous to (\ref{Sec3-IEq}), it is not difficult (for details, see \cite{GuHeYuan2014,Golsshtein1979})
to show
\begin{equation}\label{relax-ine}
\|w^{k+1}-w^*\|_G^2\leq \|w^{k}-w^*\|_G^2- \gamma(\gamma-2) \|\tilde{w}^{k+1}-w^k\|_G^2,\quad  \forall w^*\in \mathcal{M}^*.
\end{equation}
Then, based on (\ref{relax-ine}), it is straightforward to have global convergence and convergence rate analogous to Theorem \ref{global-converge}
and Theorem \ref{conv-rate}.
More precisely, combining Algorithm 1 with a relaxation step would give the following Relaxed P-PPA (RP-PPA).
\vskip4mm
\hrule\vskip2mm
\noindent {\bf Algorithm 2} (RP-PPA for solving Problem (\ref{Sec1-Prob}))
\vskip1.5mm\hrule\vskip2mm
\noindent 1\ \ Choose parameters $(\sigma, \rho, s,\tau, \varepsilon)$ satisfying (\ref{Sec3-para-region}) and $\gamma\in(0,2)$;\\
2\ \ Initialize $(x^0, y^0, \lambda^0)\in \mathcal{R}^m\times \mathcal{R}^n\times\mathcal{R}^l$ and $r^0 = Ax^0+By^0-c$;\\
3\ \ $\bar{\lambda}^0 = \lambda^{0}-\frac{\tau+\varepsilon}{s} r^0$; \\
4\ \ \textbf{For} $k=0,1,\cdots,$ \textbf{do}\\
5\ \ \quad $\widetilde{x}^k=\arg\min\limits_{x\in \mathcal{X}}\left\{f(x)+ \frac{\bar{\sigma}}{2}
\left\|A(x-x^k)-\frac{\tau}{\bar{\sigma}} \bar{\lambda}^{k} \right\|^2\right\}$;\\
6\ \ \quad  $\bar{\lambda}^{k+\frac{1}{2}}=\bar{\lambda}^k-\frac{\tau - \varepsilon}{s}\left[A(2\widetilde{x}^k-x^k)+By^k- c\right]$;\\
7\ \ \quad  $\widetilde{y}^k=\arg\min\limits_{y\in \mathcal{Y}}\left\{g(y)+ \frac{\bar{\rho}}{2}
\left\|B(y-y^k)-\frac{\tau}{\bar{\rho}} \bar{\lambda}^{k+\frac{1}{2}} \right\|^2\right\}$;\\
8\ \ \quad  $r^{k+1} = A \widetilde{x}^k + B \widetilde{y}^k -c;\ \Delta x^k = \widetilde{x}^k - x^k;\ \Delta y^k = \widetilde{y}^k -  y^{k}$;  \\
9\ \ \quad  $\widetilde{\lambda}^k=\bar{\lambda}^{k}-\frac{\tau}{s} r^{k+1}
-\frac{1}{s}\left[ \tau A \Delta x^k + \varepsilon B\Delta y^k \right]$;\\
10 \ \  Relaxation step: compute
\[
\left(\begin{array}{c}
x^{k+1}\\  y^{k+1}\\
\end{array}\right)=
\left(\begin{array}{c}
x^{k}\\  y^{k}\\
\end{array}\right) + \gamma\left(\begin{array}{c}
\Delta x^k \\ \Delta y^k\\
\end{array}\right),
\]
\ \ \ and
\[
 \bar{\lambda}^{k+1} =  \bar{\lambda}^{k} + \gamma (\widetilde{\lambda}^k - \bar{\lambda}^{k})
- \frac{(1-\gamma) (\tau+\varepsilon)}{s} (A \Delta x^k + B \Delta y^k).
\]
\hrule\vskip4mm

\section{Numerical experiments}
In this section, we would  perform some numerical
experiments for solving the following \emph{lasso} model problem arising from statistical learning
\cite{Tibshirani1996}:
\begin{equation}\label{lasso-1}
\min\limits_{x\in \mathcal{R}^n}\nu\|x\|_1+\frac{1}{2}\|Dx-b\|^2,
\end{equation}
where $\|x\|_1=\sum_{i=1}^{n}|x_i|$; $\nu>0$ is a scalar regularization parameter;
$b\in \mathcal{R}^l$ is a response vector; $D\in \mathcal{R}^{l\times n}$ is a design matrix
with $l$ and $n$ denoting the number of data points and  the number of features,
respectively. In typical applications, there are usually many more features than training examples \cite{BoydChu2010},
that is $l \ll n$, and the goal is to find a parsimonious model for the data. We refer the readers to
\cite{ChenDSa2001,Tibshirani1996,BoydChu2010,Hastie2009} for more backgrounds on the \emph{lasso} model. In the numerical experiments,
all   algorithms are coded in MATLAB 7.10(R2010a) and run
on a PC with Intel Core i5 processor (3.3GHz) with 4 GB memory.

By introducing an auxiliary variable $y\in \mathcal{R}^n$,  the problem
(\ref{lasso-1}) is obviously equivalent to
\begin{equation}\label{test-P}
\begin{array}{lll}
\min  &  \nu\|x\|_1+\frac{1}{2}\|Dy-b\|^2\\
\textrm{s.t. } &  x-y= \textbf{0},\\
     &   x\in \mathcal{R}^n,\ y\in \mathcal{R}^n,
\end{array}
\end{equation}
which is a special case of (\ref{Sec1-Prob}) with
\[
f(x)=\nu\|x\|_1,\ g(y)=\frac{1}{2}\|Dy-b\|^2,\ A=\textbf{I},\ B=-\textbf{I},\ c=\textbf{0}.
\]
Applying Algorithm 1 to solve (\ref{test-P}), we have
\[
x^{k+1}=\arg\min\limits_{x\in\mathcal{R}^n}\left\{\nu\|x\|_1
+\frac{\bar{\sigma}}{2}\left\|x-x^k-\frac{\tau}{ \bar{\sigma}}\bar{\lambda}^{k}\right\|^2\right\},
\]
for which $x^{k+1}$ can be explicitly obtained by a soft shrinkage operator
\cite{{Donoho Tsaig2008}} and $\bar{\sigma} $ is defined in (\ref{sigma-rho-bar}).
In addition, we can deduce that
\[
y^{k+1}=\left[D^TD+ \bar{\rho}\textbf{I}\right]^{-1}
\left[ D^Tb+\bar{\rho}\left(y^k-\frac{\tau}{ \bar{\rho}} \bar{\lambda}^{k+\frac{1}{2}} \right)\right],
\]
where $\bar{\rho}$ is defined in (\ref{sigma-rho-bar}).
Notice that the matrix $ D^TD+\bar{\rho}\textbf{I}$ is positive definite,
since $\bar{\rho}=(\rho s+\tau^2-1)/s >0$. Though it maybe quite time consuming
to compute $(D^TD+\bar{\rho}\textbf{I})^{-1}$ and $D^Tb$ when the problem scale
is large, they only need to be computed once before the iteration starts.
To save computation, we can actually compute once and cache the Cholesky Factorization of
the much smaller matrix $DD^T/\bar{\rho}+\textbf{I}$ (note $l\ll n$), which
 takes about $(l^2n+l^3/3)$ flops, including the cost of forming $DD^T$
and the Cholesky Factorization. Then, all the subsequent $y$-updates can
be calculated by the Sherman-Morrison inversion matrix formula together with forward-backward substitutions.

The problem data are generated by the following way.
Each entry of the feature matrix $D$ is draw from the standard normal distribution $\mathcal{N}(0,1)$ and then
each of its column  is normalized.  A random sparse vector $x^{true}\in \mathcal{R}^n$ with
100 nonzero entries is also drawn from an $\mathcal{N}(0,1)$ distribution. The vector $b$ is computed via
$b=Dx^{true}+\epsilon$, where $\epsilon\sim \mathcal{N}(0,10^{-3}\textbf{I})$, and the regularization
parameter is set as $\nu=0.12\nu_{max}$ with $\nu_{max}=\|D^Tb\|_\infty$.
The following stopping criterion is used for all the comparison algorithms
\[
\textrm{IRE(k)}:=\frac{\left\|x^k-y^k\right\|}{\max\{\left\|x^k\right\|,\left\|y^k\right\|\}}
\leq \textrm{Tol}\quad \mbox{and} \quad \frac{\phi(u^k)-\phi^*}{\phi^*}\leq 1.0\times 10^{-8},
\]
where \textrm{Tol} is a given tolerance, $\phi(u^k)=f(x^k)+g(y^k)$, and  $\phi^*$ is the approximate
optimal objective function value obtained by running P-PPA after $2000$ iterations.
Then, all the comparison algorithms are set to have the maximum $2000$ number of iterations and
they all use the same starting point  $(x^0,y^0,\lambda^0)=(\textbf{0},\textbf{0},\textbf{0})$.

\subsection{Effects of parameters $(\sigma,\rho,s,\tau,\varepsilon)$}
The aim of this subsection is to investigate   how the five parameters $(\sigma,\rho,s,\tau,\varepsilon)$ would influence the performance of P-PPA. For this purpose, we first fix the free parameters $(\tau,\varepsilon)$ as $(3, 1.5)$ and then change other distinctively constrained parameters $(\sigma,\rho,s)$ to
  investigate their effects  on  P-PPA.

Table 1 presents the numerical results of Algorithm 1 (i.e. P-PPA) with different parameters for solving
the test problem (\ref{test-P}) with dimension $(l,n)=(1800,4000)$. For this set of tests, we fix
the tolerance $\textrm{Tol}=1.0\times10^{-6}$. And in all the numerical Tables, ``Iter", ``CPU" and
``DRN"  denote the   iteration numbers, the CPU time in seconds and  the dual residual
norm $\|y^{k}-y^{k-1}\|$ , respectively.
We can  observe from Table 1 that:
\begin{itemize}
\item For  the parameters $(\sigma,\rho,s)$, the reported results in each column of IRE, DRN and $\phi(u^k)$ are nearly the same when fixed any two parameters with one parameter changing.
\item With the increase of the parameter $\sigma$ or $\rho$, both the
iteration number and the CPU time tend to increase(denoted by $\downarrow$);
\item   With the increase of the parameter $s$, both the iteration number and
the CPU time decrease firstly and then increase(denoted by $^\uparrow_\downarrow$).
\end{itemize}
These   changing trends   identify  with Remark \ref{agaga}. Reported results of Table 1 indicate that the choice of the parameters $(\sigma,\rho,s)$
could have a great effect  on the performance of P-PPA,  the value in the place of the arrow is   better than others in each subtable, and it seems that setting
$(\sigma,\rho,s)=(0.8,6,3)$ would be a reasonable choice for solving the
test problem (\ref{test-P}).

\begin{remark} \label{rem1}
{\color{blue}Noting from Table 1 that   if any four parameters are fixed, then by (\ref{Sec3-para-region}) the remaining  one is clearly subjected to a given domain. For instance, in the top subtable of Table 1 we have from (\ref{Sec3-para-region}) that $\sigma>\frac{37.25}{51}\approx0.73.$
Therefore, we can randomly choose some  values in such region to  do experiments to find
out which one gives relatively better performance.}
\end{remark}

\vskip2mm
\begin{center}
\begin{tabular}{cccccc}
\hline
 Parameters &  Iter  & CPU & IRE  &  DRN  & $\phi(u^k)$  \\
\hline
$\sigma(\rho=6,s=3)$\\ \hline
 0.8 & $\textbf{134}_\downarrow$ &$\textbf{3.60}_\downarrow $&9.8125e-7 &1.0932e-5&19.2402\\
 1 & 137 &  3.70  &  9.3824e-7  & 1.0508e-5 &  19.2402   \\
 2 &  149 & 4.00    &   9.3133e-7  &  1.0673e-5 &19.2402      \\
 4 & 174 & 4.49    &   8.5235e-7  &  1.0120e-5 &  19.2402    \\
 6 & 199 & 5.08    &   7.8227e-7  &  9.5316e-6 &  19.2402    \\
 8 & 223 &5.55     &   7.5285e-7  &  9.3548e-6 &  19.2402    \\
 10 & 248 &  6.25   &   6.9449e-7  &  8.7643e-6 & 19.2402     \\
\hline
$\rho(\sigma=6,s=3)$\\ \hline
 0.8 & $\textbf{137}_\downarrow$ &$\textbf{3.61}_\downarrow $&6.2401e-7 &8.0288e-6&19.2402\\
 1 & 140     & 3.70     &  6.0813e-7    &   7.8031e-6   &  19.2402   \\
 2 & 152 &   3.94   &  6.5624e-7    &   8.3108e-6   &  19.2402   \\
 4 &  176   & 4.56     &  7.1727e-7    &   8.8891e-6   &   19.2402  \\
 7 & 210 & 5.22     &  8.2280e-7    &   9.9542e-6   &  19.2402   \\
 8 & 222 & 5.47     &  8.2070e-7    &   9.8666e-6   &  19.2402   \\
 11 & 255 & 6.22     &  8.9283e-7    &   1.0567e-5   &  19.2402   \\
\hline
$s(\sigma=6,\rho=6)$\\ \hline
 1 &  263&     6.48   &   9.7027e-7   &   1.4122e-5   &    19.2402  \\
 2 &  211& 5.25  &   9.8606e-7   &   1.3624e-5   & 19.2402     \\
 5 &  197 & 4.99  & 4.3585e-7   & 3.8039e-6   & 19.2402 \\
 7 &  192 & 4.84  & 3.9519e-7   & 2.4871e-6   & 19.2402 \\
  11 & $\textbf{181}_\downarrow^\uparrow$   & $\textbf{4.67}_\downarrow^\uparrow$  &   9.9008e-7    &   2.9591e-6   &  19.2402    \\
   12 & 195   & 4.86       &  8.9998e-7   &   2.8168e-6   & 19.2402     \\
  14 & 221   & 5.47       &   9.6407e-7   &   3.3572e-6   & 19.2402     \\
  16 & 255    & 6.39 &    8.7023e-7   &    2.4315e-6   & 19.2402     \\
  18 &  282 &  6.92 &    8.5695e-7   &    2.7583e-6   & 19.2402     \\
\hline
\end{tabular}
\end{center}
\begin{center}\small
Table 1:\ Results of problem (\ref{test-P})  by  P-PPA with different parameters $(\sigma,\rho,s)$.
\end{center}

\begin{remark}
{\color{blue} In a similar  way as  mentioned in Remark \ref{rem1}, we have
\[
\tau^2<\frac{(\sigma s-1)(\rho s-1)}{\varepsilon^2}\quad \textrm{and} \quad \varepsilon^2<\frac{(\sigma s-1)(\rho s-1)}{\tau^2}.
\]
Then, by testing some values of the parameter $\tau(\varepsilon)$ and observing which one performs approximately better, our tuned results are $(\tau,\varepsilon)=(3,1.5).$ Hence, for  further comparative experiments with some state-of-the-art methods, we would use
the tuned results $(0.8,6,3,3,1.5)$ as the default parameter setting for both P-PPA and RP-PPA.}
\end{remark}

\subsection{Comparative experiments}
Now, we would like to compare P-PPA and RP-PPA with  other two popular  methods
for solving the  problem (\ref{test-P}):
ADMM\footnote
{Available at http://web.stanford.edu/$\sim$boyd/papers/admm/.} \cite{BoydChu2010}
 and R-PPA \cite{GuHeYuan2014}.

\begin{center}
\begin{tabular}{ccccc}
\hline
P-PPA$(l, n)$&  Iter  & CPU & IRE  & $\phi(u^k)$  \\
\hline
(1000,4000)  & 313& 4.01   &  9.7448e-11  & 19.3556  \\
(1800,4000)  &  265 & 6.27     &  9.5478e-11  &  19.2402 \\
(1000,10000) & 387 & 11.25 &  9.6822e-11  &19.0072 \\
(1800,10000) & 260 & 13.16  &  9.6308e-11  & 20.0529 \\
(1000,16000) & \textbf{279} &\textbf{11.71} &  9.8237e-11  &  18.6698  \\
(1600,16000) &  247 & 17.24  &  9.5727e-11  & 19.7810\\
(1000,20000) & \textbf{316} &\textbf{16.43}  &  9.2613e-11  & 19.9294 \\
(1800,20000) & \textbf{196}   & \textbf{19.80}   & 8.9340e-11  & 19.2038 \\
(1000,24000) & \textbf{369}& \textbf{22.62} &  9.8638e-11  & 18.5250 \\
(2000,26000) & \textbf{212} &  \textbf{29.26} &  9.4504e-11  & 19.0005  \\
\hline
RP-PPA$(l, n)$&  Iter  & CPU & IRE  & $\phi(u^k)$  \\
\hline
(1000,4000)  & 260 & 3.28    &   9.3533e-11  &  19.3556  \\
(1800,4000)  & 219 & 5.20    &   9.7552e-11  & 19.2402   \\
(1000,10000) & 322 &  9.45  &   9.4616e-11  & 19.0072   \\
(1800,10000) & 216 &  11.09  &   9.2376e-11  &   20.0529 \\
(1000,16000) & \textbf{232} &  \textbf{10.00}   &   9.2286e-11  &   18.6698 \\
(1600,16000) &  206&   14.49  &   9.5818e-11  &  19.7810  \\
(1000,20000) & \textbf{259} &  \textbf{13.70}   &   9.3551e-11  &  19.9294  \\
(1800,20000) &  \textbf{173}&  \textbf{17.58}   &   9.1322e-11  & 19.2038   \\
(1000,24000) &  \textbf{302}&   \textbf{18.64}  &   9.6840e-11  &   18.5250  \\
(2000,26000) & \textbf{174} &   \textbf{24.40} &   9.5930e-11  &   19.0005 \\
\hline
R-PPA$(l, n)$ &  Iter  & CPU & IRE &$\phi(u^k)$  \\
\hline
(1000,4000)  &  284 & 3.67  &   9.7350e-11  &  19.3556  \\
(1800,4000)  & 238 & 5.68 &  9.3094e-11  & 19.2402 \\
(1000,10000) &371 & 10.85  &  9.9245e-11  & 19.0072 \\
(1800,10000) & 243  & 12.27  &  9.5912e-11  & 20.0529  \\
(1000,16000) & 336& 13.99  &  9.8769e-11  &  18.6698     \\
(1600,16000) & 256 & 17.74 &  9.5437e-11  &  19.7810 \\
(1000,20000) & 339 &17.47 & 9.3942e-11  & 19.9294 \\
(1800,20000) & 236 & 22.80 &9.9142e-11  &19.2038 \\
(1000,24000) & 374&22.92 &  9.7664e-11  & 18.5250 \\
(2000,26000) & 226  & 31.13   &  9.7351e-11  & 19.0005  \\
\hline
ADMM$(l, n)$&  Iter  & CPU & IRE  &$\phi(u^k)$  \\
\hline
(1000,4000)  &\textbf{100} & \textbf{1.36}  &   9.2692e-11  & 19.3556   \\
(1800,4000)  & \textbf{71} & \textbf{2.16} &  9.2988e-11  &  19.2402 \\
(1000,10000) & \textbf{189} &\textbf{5.62}  &  9.1517e-11  & 19.0072\\
(1800,10000) & \textbf{126} &\textbf{6.80}   &  8.9039e-11  & 20.0529 \\
(1000,16000) & 285 & 11.95 &  9.8136e-11  &  18.6698     \\
(1600,16000) & \textbf{193} &  \textbf{13.69} &  9.2166e-11  & 19.7810 \\
(1000,20000) & 351 &18.04  &  9.1173e-11  & 19.9294 \\
(1800,20000) &208 & 20.33 &9.9380e-11  & 19.2038\\

(1000,24000) & 426 & 26.06 &  9.5935e-11  &18.5250  \\
(2000,26000) & 236  & 32.28   &  9.7005e-11  & 19.0005  \\
\hline
\end{tabular}
\end{center}
\begin{center}\small
Table 2:\ Comparative results of problem (\ref{test-P}) with different dimensions\footnote{The bold value excluding that of P-PPA is of the smallest in each experiment with respect to (l,n), and the bold value of P-PPA is smaller than that of R-PPA and ADMM.}.
\end{center}

Table 2 reports the numerical results of all comparison methods for
solving the problem (\ref{test-P}) with different dimensions $(l,n)$.
 Actually, ADMM  uses   the downloaded codes but with  penalty parameter 1 and a widely used step-length 1.618 when updating the lagrangian multipliers.
For R-PPA, the penalty parameter is set as $10$, which is reasonably
good for this method. Both RP-PPA and R-PPA use the same relaxation factor
$\gamma = 1.2$. The tolerance $\textrm{Tol}=1.0\times10^{-10}$ is
used for all the testing problems in Table 2.
The numerical results of solving the test problem with fixed dimensions
 $(l,n)=(1800,20000)$, but under different accurate tolerances, are presented in Table 3.
 Moreover,  the comparative convergence curves of the objective function $\phi(u^k)$, the residual error IRE(k)
and the dual residual norm $\|y^k-y^{k-1}\|$ against the number of iterations
are shown in Fig. 1.

\begin{center}
\begin{tabular}{ccccc}
\hline
$\textrm{Tol}=10^{-5}$& P-PPA  & RP-PPA  & R-PPA  &ADMM  \\
\hline
Iter &     100  &\textbf{86}  &  102   & 88 \\
CPU  &  10.53  & \textbf{9.31}& 10.68   & 9.54 \\
IRE  & 9.9112e-6 & 9.0542e-6& 9.5777e-6  & 9.2727e-6 \\
$\phi(u^k)$& 19.2038   &19.2038 &  19.2038      &  19.2037 \\
\hline
$\textrm{Tol}=10^{-8}$& P-PPA & RP-PPA   & R-PPA  &ADMM  \\
\hline
Iter & 159 & \textbf{137} &    182   &  158  \\
CPU   & 16.80  &\textbf{14.27} &  18.64   & 16.65 \\
IRE  &   8.9178e-9 & 9.7009e-9&   9.5103e-9     &  9.3515e-9 \\
$\phi(u^k)$& 19.2038  & 19.2038 & 19.2038       & 19.2038  \\
\hline
$\textrm{Tol}=10^{-11}$& P-PPA & RP-PPA   & R-PPA  &ADMM  \\
\hline
Iter &  \textbf{214} & \textbf{190} &  264     & 234  \\
CPU   & \textbf{21.72} &\textbf{18.48} & 27.03 & 24.08   \\
IRE  &  9.2468e-12  & 9.8367e-12&   9.5013e-12    & 9.5918e-12  \\
$\phi(u^k)$& 19.2038  & 19.2038 & 19.2038       & 19.2038  \\
 \hline
 $\textrm{Tol}=10^{-14}$& P-PPA  & RP-PPA  & R-PPA  &ADMM  \\
\hline
Iter &  \textbf{274} & \textbf{244} &  347   &  2000 \\
CPU   & \textbf{26.79} &\textbf{24.18}  & 32.24   &192.33   \\
IRE  &  9.9348e-15  & 8.9322e-15 &9.7695e-15      & 1.3963e-14  \\
$\phi(u^k)$& 19.2038   &19.2038 &  19.2038      & 19.2038  \\
\hline
\end{tabular}
\end{center}
\begin{center}\small
Table 3:\ Comparative results of problem (\ref{test-P}) under different tolerance errors\footnote{The bold value of P-PPA is smaller than that of R-PPA and ADMM, and the bold value of RP-PPA is of the smallest.}.
\end{center}

\begin{figure}[htbp]
 \begin{minipage}{1\textwidth}
 \def\figurename{\footnotesize Fig.}
 \centering
\resizebox{12.5cm}{13cm}{\includegraphics{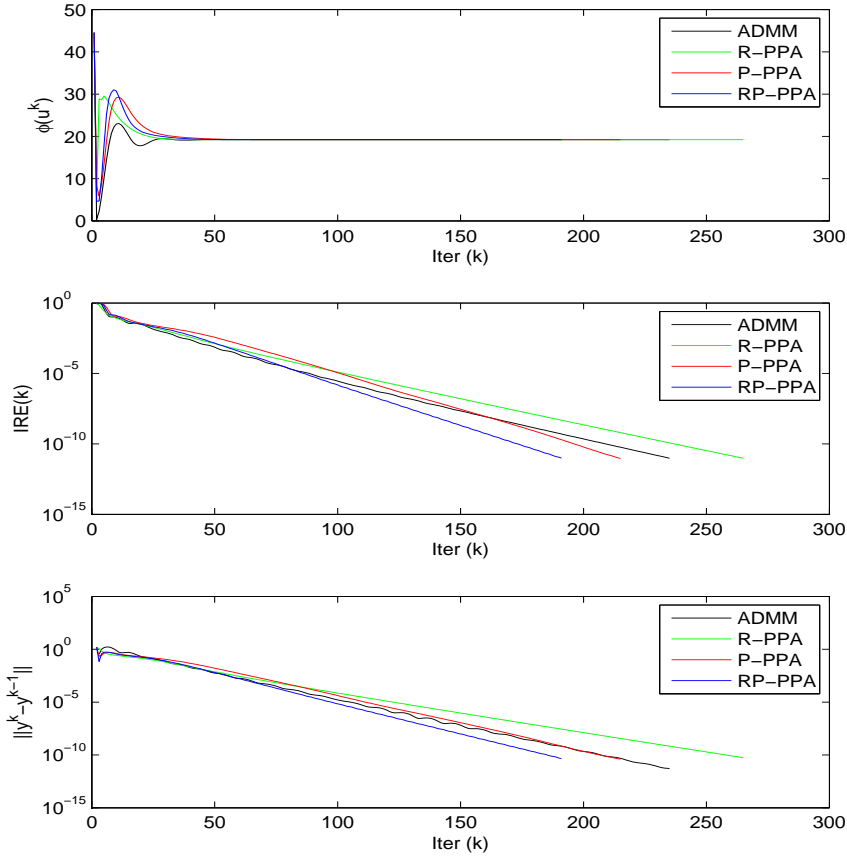}}
\caption{\footnotesize  Comparative convergence curves of the objective function
$\phi(u^k)$, the residual  IRE(k) and the dual
residual norm $\|y^k-y^{k-1}\|$ against the number of iterations under tolerance} $\textrm{Tol}=1.0\times10^{-11}$.
   \end{minipage}
\end{figure}

We can observe from Tables 2-3 that both P-PPA and RP-PPA  perform
better than ADMM and R-PPA  for the relatively large size problems
 in terms of both the number of iterations and the CPU time. {\color{blue} Another outstanding observation is that RP-PPA can  clearly shorten  the number of iterations and the CPU time of P-PPA.  Besides, ADMM performs
 the worst for  large-size problems, while it performs the best for small-size problems (e.g. $n=4000$). Both Table 3 and Fig. 1  illustrate that
as the tolerance becomes smaller,  P-PPA and RP-PPA could perform significantly
better than  ADMM and R-PPA.} In addition, note from Table 3 that ADMM fails to solve the problem
with dimensions $(l,n)=(1800,20000)$ in $2000$ iterations to achieve the accuracy $\textrm{Tol}=1.0\times10^{-14}$. {\color{blue}Reported results show that
our  proposed methods are  efficient
when properly choosing the algorithmic parameters.}

\begin{remark} {\color{blue} Although the tuned values of the parameters  in the proposed algorithms are not proved theoretically to be the best, the reported numerical results of  comparative experiments are sufficient to show that such a choice can make our algorithms outperform the other two algorithms.}
\end{remark}

\section{Conclusion and discussion}
By introducing several parameters to the proximal matrix in the framework of the traditional
proximal point algorithm, we propose a new Parameterized Proximal Point Algorithm
(P-PPA) for solving the separable convex minimization problem.
Under certain conditions on these parameters, we show that the P-PPA is globally
convergent and would maintain a  worst-case $O(1/t)$ ergodic  convergence rate.
By properly choosing the parameters, the numerical experiments of solving the classical
\emph{lasso} problem in statistical learning indicate our P-PPA and the Relaxed
P-PPA (RP-PPA) could perform significantly better than the other two
benchmark methods: ADMM and R-PPA, especially
for solving large scale problems and high accurate solutions are required.

{\color{blue} For the case that the subproblems are not easy to solve,  inexact ADMMs\cite{HagerZhang2016} are recently developed for solving the general separable convex optimization problems with a linear constraint and with an objective including  smooth plus nonsmooth terms, which  is particularly useful when the ADMM's subproblems do not have closed solutions or when the solution of the subproblem is expensive.
Also, there are  other works in references \cite{He2013,Solodov2000}
which discuss how to solve the subproblems inexactly.}

Finally, observe that the P-PPA and RP-PPA can be naturally extended to
 to solve the problem (\ref{Sec1-Prob}) with inequality constraints
or with matrix variables such as
\begin{equation} \label{matrix-P}
\begin{array}{lll}
\min  &  f(X)+g(Y)\\
\textrm{s.t. } &   AX+BY=C,\\
     &   X\in \mathcal{X}, Y\in \mathcal{Y},
\end{array}
\end{equation}
where both $f$ and $g$ are proper  closed convex functions over the matrix variables
$X$ and $Y$, $A$ and $B$ are coefficient matrices, $\mathcal{X}$ and
$\mathcal{Y}$ are certain closed convex sets. The model problem (\ref{matrix-P})
also arises very often in many important applications in data analysis \cite{CandWright2011,LiuLiBaiLiu2017},
for example, the robust principal component analysis in image processing, etc.
One of the following research tasks could be to extend the P-PPA to solve the
multi-block separable convex/nonconvex programming problems.

\vskip 3mm \noindent{\large\bf Acknowledgements}\vskip 2mm

The authors wish to thank  the
Editor-in-Chief Prof. O.A. Prokopyev  and the  anonymous referees  for providing their valuable suggestions which have significantly improved
the quality of the paper.



\end{document}